\documentclass[10pt]{article}
\usepackage{amssymb, amsmath, latexsym, a4}
\usepackage[latin1]{inputenc}
\usepackage{graphicx}
\usepackage[all]{xy}
\usepackage[usenames]{color}
\usepackage{pgf,tikz}

\usepackage{mathrsfs}

\usetikzlibrary{arrows}

\newenvironment{proof}[1][Proof]{\noindent\textbf{#1.} }{\ \rule{0.5em}{0.5em}}

\newtheorem{De}{Definition}[section]
\newtheorem{Th}[De]{Theorem}
\newtheorem{Pro}[De]{Proposition}
\newtheorem{Le}[De]{Lemma}
\newtheorem{Co}[De]{Corollary}
\newtheorem{Rem}[De]{Remark}
\newtheorem{Ex}[De]{Example}

\def\d{\delta}

\def\s{\sigma}

\def\ot{\otimes}

\def\lra{\longrightarrow}

\def\mapright#1{\smash{\mathop{\lra}\limits^{#1}}}

\def\NM{{\mathbf N}}

\newbox\pullbackbox
\setbox\pullbackbox=\hbox{\xy 0;<1mm,0mm>: \POS(4,0)\ar@{-} (0,0)
\ar@{-} (4,4)
\endxy}

\newdir{>>}{{}*!/3.5pt/:(1,-.2)@^{>}*!/3.5pt/:(1,+.2)@_{>}*!/7pt/:(1,-.2)@^{>}*!/7pt/:(1,+.2)@_{>}}
\newdir{ >>}{{}*!/8pt/@{|}*!/3.5pt/:(1,-.2)@^{>}*!/3.5pt/:(1,+.2)@_{>}}
\newdir{ |>}{{}*!/-3.5pt/@{|}*!/-8pt/:(1,-.2)@^{>}*!/-8pt/:(1,+.2)@_{>}}
\newdir{ >}{{}*!/-8pt/@{>}}
\newdir{>}{{}*:(1,-.2)@^{>}*:(1,+.2)@_{>}}
\newdir{<}{{}*:(1,+.2)@^{<}*:(1,-.2)@_{<}}

\begin{document}

\title{Posets, their Incidence Algebras and Relative Operads, and the Cohomology Comparison Theorem}
\author{Batkam Mbatchou V. Jacky III$^{1}$, Fr\'ed\'eric Patras$^{2}$,  Calvin
Tcheka$^{3}.$} \maketitle

\bigskip

\centerline{$^{1}$ Department of Mathematics, Faculty of
    science-University of Dschang} \centerline{Campus Box (237)67
    Dschang, Cameroon} \centerline{ {E-mail address}: batkamjacky3@yahoo.com }
    
\bigskip

\centerline{$^{2}$ Universit\'e C\^ote d'Azur et CNRS} \centerline{UMR 7351 LJAD, Parc Valrose, 06000 Nice, France} \centerline{ {E-mail address}: patras@unice.fr}

\bigskip

\centerline{$^{3}$ Department of Mathematics, Faculty of
science-University of Dschang} \centerline{Campus Box (237)67
Dschang, Cameroon} \centerline{ {E-mail address}:
calvin.tcheka@univ-dschang.org}

\bigskip

\date{}

\bigskip

 \centerline{\bf Abstract} Motivated by various developments in algebraic combinatorics and its applications, we investigate here the fine structure of a fundamental but little known theorem,
the Gerstenhaber and Schack cohomology comparison theorem.
The theorem classically asserts
that there is a cochain equivalence between the usual singular cochain complex of a simplicial complex and the
relative Hochschild complex of its incidence algebra, and a quasi-isomorphism with the standard Hochschild complex. Here, we will be mostly interested in its application to arbitrary posets (or, equivalently, finite topologies) and their incidence algebras.
We construct various
 structures, classical and new, on the above two complexes: cosimplicial, differential graded algebra, operadic and brace algebra structures and show that the comparison
 theorem preserves all of them. These results provide non standard insights on links between the theory of posets, incidence algebras, endomorphism operads and finite and combinatorial topology. By {\it non standard}, we refer here to the use of {\it relative} versions of Hochschild complexes and operads.
\\

\bigskip


\bigskip

\bigskip

\section{ Introduction}
Various recent results involve or develop connexions between posets, finite topologies, operads, brace algebras, coalgebras, Hopf algebras and algebraic topology.
 Among those, we may mention here as directly relevant for the present work, Foissy's works on the combinatorial and brace algebras structures on operads
 \cite{Foissy}, advances on the algebraic structures on finite topologies, posets and quasi-posets \cite{F5,F4,F3,F1} and various works in the combinatorics
 of free probability, among which \cite{FK}.

The purpose of the present paper is to get new insights on these
connexions using a theorem of Gerstenhaber and Schack (the
cohomology comparison theorem, CCT). The CCT asserts that, for a
given triangulated topological space with associated incidence
algebra A, there exists a quasi-isomorphism of cochain complexes
between a certain relative cohomological Hochschild complex of A and
the singular cochain complex of that topological space \cite{6, 7},
and a quasi-isomorphism with the standard Hochschild complex
\cite{6}, \cite[Th. 1.3]{3}. When extended to arbitrary posets and
finite topologies, the theorem allows to connect two families of
objects of different nature: an operadic-type object of a non
standard type (the relative Hochschild complex of an incidence
algebra) and a classical object in combinatorial topology (the
singular cochain complex associated to the nerve of a poset). Our
results investigate how algebraic structures on these objects
transport under the CCT: differential brace algebras (following an
unpublished paper by one of us \cite{fp}) and various algebraic and
topological structures associated to operads (that are investigated
more generally in recent work by two of us \cite{2}). The
equivalence of categories between finite ($T_0$) topologies and
posets is the prototype for more complex duality phenomena such as
Stone duality. We hope our results will also lead to developments in
that direction.

The article also features the meaningfulness of the notion of relative calculus and operads, introduced here and abstracted from the notion of relative Hochshicld cochains by extending the scope of the usual links between Hochschild complexes and endomorphism operads.

The paper is organized as follows: Section \ref{poset} recollects basic constructions on posets (incidence algebra, nerve, cochain algebra structure). Section \ref{Hochschild} develops various constructions around the incidence algebra of a poset, recalls the notion of relative Hochschild cochains and the CCT. In the process constructions underlying the CCT are introduced at the dual (chain, homological) level and a dual homological comparison theorem is proved. Section \ref{cosimpl} extends the CCT at the cosimplicial level. Section \ref{opera} introduces an operadic structure on relative Hochschild cochains and shows that the CCT holds at operadic level. Section \ref{relop} features the idea that constructions on posets are best understood using relative constructions. This is illustrated by the introduction of relative operads, a notion also meaningful for operator-valued probability. The article concludes by showing that the CCT also holds at the level of brace differential graded algebras.

\begin{Rem} We conclude this introduction with a remark on finite topologies. The article will be mostly written in the language of posets, however its deeper meaning is best understood when having in mind the connexion between the latter and finite topologies. There is indeed a bijection between finite topological spaces $\mathcal T$ satisfying the $T_0$ separation axiom (given two points $x,y$ in $\mathcal T$ there should always exist an open set that contains one point but not the other) and finite posets. Any finite topological space is (finitely) homotopy equivalent to a finite $T_0$ space, so that the separation assumption is not a serious restriction. It follows in particular that all constructions in the article have a direct topological meaning, besides a combinatorial one. For a survey of the theory of finite topological spaces and of the bibliography in the subject, we refer to \cite{F5}.
\end{Rem}

{\bf Acknowledgements}: FP acknowledges
support from the ANR project Algebraic Combinatorics, Renormalization, Free probability and
Operads -- CARPLO (Project-ANR-20-CE40-0007) and from the ANR -- FWF project PAGCAP.

\subsection*{Conventions and notation}
 In the sequel,

 \noindent
$\bullet$ $\mathbb{K}$ denotes an arbitrary field. Vector spaces, tensor
products and linear maps are defined over $\mathbb{K}$ unless
otherwise stated.\\
$\bullet$ The linear span of a set $X$ is denoted ${\mathbb K}(X)$, the cardinality of $X$ is denoted $|X|$.\\
$\bullet$ If $V$ is a graded object,
then  the suspension of V is defined as follows: $(sV )_n = V_{n-1}$.\\
  An element $x\in V_n$ is of degree $n$ and
  we write $|x|= n$.\\
$\bullet$ If $A\stackrel{f}\rightarrow B$ is a vector space map then the dual map is denoted: $B^{\ast}\stackrel{f^{\ast}}\rightarrow A^{\ast}$.\\
$\bullet$ The Kronecker symbol is written $\partial_{x,y}$ (recall it is equal to 1 if $x=y$ and 0 else).\\

\section{Posets, nerves, chains and cochains}\label{poset}

Let $(\Sigma ,\leq)$ be a finite poset (that is, a finite ordered set, where the order is not strict in general). An increasing sequence in $\Sigma$,  $(\sigma_0\le\cdots\le \sigma_n)$, is called a $n$-chain (or simply a chain). It is a strict (or nondegenerate) chain if and only if the sequence is strictly increasing ($\sigma_0<\cdots < \sigma_n$).

Recall that the category $\Delta$ is the category whose objects are the $[n]:=\{0,\dots,n\}$ and whose morphisms are the weakly increasing maps between them. The set of morphisms of $\Delta$ is generated (under the composition product) by the face maps $d_i^n:[n-1]\to[n],\ 0\leq i\leq n$ and the degeneracy maps $s_i^n:[n+1]\to [n], \ i=0,\dots,n$. Here, $d_i^n$ is the unique injective weakly increasing map from $[n-1]$ to $[n]$ such that $i$ is not in its image and $s_i^n$ the unique surjective weakly increasing map that maps $i$ and $i+1$ to $i$. The face and degeneracy maps satisfy identities such as $$s_{j-1}^n\circ s_i^{n+1}=s_i^n\circ s_{j}^{n+1},\ 0\leq i<j\leq n+1$$ that can be used alternatively to define simplicial sets, see
\cite{May} for details.

A simplicial (resp. cosimplicial) set is a contravariant (resp. covariant) functor from $\Delta$ to ${\bf Set}$, the category of sets. The nerve of $\Sigma$, denoted  $\hat{\Sigma}$, is the simplicial set whose n-simplices are ordered
 morphisms (weakly increasing maps) from $[n]$ to $\Sigma$ or equivalently the chains in $\Sigma$,
 $\sigma_0\le\cdots\le \sigma_n$. As some authors use different conventions, notice that, in our terminology, chains correspond to weakly increasing maps from $[n]$ to $\Sigma$ whereas strictly increasing maps correspond to strict chains.

\begin{Rem}
A particular case is the one
where $\Sigma$ is a subset of the set of subsets of
an arbitrary non empty finite set $S$. Then,  $\Sigma$ is a finite
simplicial complex if, for each $\sigma\in \Sigma,$ the
set of non empty subsets of $\sigma$ is a
subset of $\Sigma.$ If $|\sigma| = n + 1,$ $\sigma$ is called a
n-simplex of $\Sigma.$ Elements of $\Sigma$ are ordered by inclusion and
in particular $\Sigma$ endowed with the inclusion order denoted
$\leq,$ can be viewed as a poset. In that case, the simplicial set $\hat{\Sigma}$ is also
named barycentric subdivision of $\Sigma$, its homology and cohomology compute the simplicial homology and cohomology of $\Sigma$.\end{Rem}

The simplical structure on $\hat{\Sigma}:=
\{\hat{\Sigma}_n\}_{n\ge 1}$ is constructed as follows: for $n\ge 1$ and $0\le i\le
n,$
\begin{enumerate}
    \item  The face morphism ${d}^{n}_i$, from $\hat{\Sigma}_n$ to $\hat{\Sigma}_{n-1}$ is defined by
    $$\begin{array}{ccc}
        \hat{\Sigma}_n &\stackrel{{d}^{n}_i}\longrightarrow & \hat{\Sigma}_{n-1} \\
        A & \longmapsto & (\sigma_0\le\sigma_1\le\sigma_2\le\cdots
    \le \sigma_{i-1}\le\sigma_{i+1}\le \cdots\le \cdots\le\sigma_n)
    \end{array}$$
    for $A:=(\sigma_0\le\sigma_1\le\sigma_2\le\cdots
    \le \sigma_{i-1}\le\sigma_i\le\sigma_{i+1}\le \cdots\le \cdots\le\sigma_n)$
or equivalently ${d}^{n}_i(A)=(\sigma_0\le\sigma_1\le\cdots\le\sigma_{i-1}\le\hat{\sigma}_i\le\sigma_{i+1}\le\cdots\le\sigma_n),$
    where the notation $\hat{\sigma}_i$ means that the vertex is  omitted.
    \item
    The degeneracy morphism $s^{n}_i$ duplicates   the simplex $\sigma_i$ at the position i of the sequence:
    $$\begin{array}{ccc}
        \hat{\Sigma}_n &\stackrel{s^{n}_i}\longrightarrow & \hat{\Sigma}_{n+1} \\
        A & \longmapsto & (\sigma_0\le\sigma_1\le\cdots\le\sigma_{i-1}\le\sigma_i\le\sigma_i\le\sigma_{i+1}\le\cdots\le\sigma_n)
    \end{array}$$

\end{enumerate}
    It is an easy and standard exercise to show that these face and degeneracy maps define a simplicial set structure on $\hat\Sigma$ \cite{May}.

 The simplicial vector space $\mathbb K(\hat\Sigma)$ generated by $\hat\Sigma$ is denoted $C_\ast(\hat\Sigma)$ and called the singular chain complex of $\Sigma$. It is obtained by applying the functor $X\to \mathbb K(X)$ to $\hat\Sigma$.

 Recall (see e.g. \cite{May}) that

 \begin{Pro}
  The singular chain complex, $C_*(\hat{\Sigma})=\{C_n(\hat{\Sigma})\}_{n\ge 0},$ is naturally endowed by the Alexander-Whitney map with
a differential graded coalgebra structure.
\end{Pro}

The coproduct is given by: $$\begin{array}{ccc}
C_n(\hat{\Sigma})&\stackrel{\Delta}\longrightarrow &\bigoplus\limits_{j=0}^nC_j(\hat{\Sigma})
\otimes C_{n-j}(\hat{\Sigma})\\
u & \longmapsto & \Delta(u)=\sum\limits_{j=0}^{n}\Delta_{j,n-j}(u)=
\sum\limits_{j=0}^{n}\widetilde{{d}}_{n}^{n-j}(u)\otimes\widetilde{{d}}_0^j(u)
\end{array}$$
    where $\widetilde{{d}}^{n-j}_{n}=\underbrace{d^{j+1}_{j+1} d^{j+2}_{j+2}\cdot \cdots  d^{n}_n}_{n-j times}$
    and $\widetilde{d}_0^{j}=\underbrace{d^{n-j+1}_0\cdots
    d^{n}_0}_{j times}$ with $\widetilde{d}_{0}^{0}=Id=\widetilde{d}_n^{0}$\\ and the boundary operator by:
    $\partial_n=\sum\limits_{i=0}^{n}(-1)^{i}d^{n}_i.$
The coaugmentation map is the map that sends all $0$-simplices to 1 and higher dimensional simplices to 0.

Dualising, one gets a cosimplicial structure on the singular cochain complex
$C^*(\hat{\Sigma})=\{C^n(\hat{\Sigma}):=Hom(\hat{\Sigma}_{
n},\mathbb{K})\}_{n\ge 1}$ whose coface and codegeneracy maps are:
   $$\begin{array}{ccc}
        C^{n-1}(\hat{\Sigma})&\stackrel{\mathfrak{F}^{n-1}_i}\longrightarrow & C^{n}(\hat{\Sigma}) \\
        f& \longmapsto & \mathfrak{F}^{n-1}_i(f)=f\circ d^{n}_i.
   \end{array}$$

$$\begin{array}{ccc}
        C^n(\hat{\Sigma})&\stackrel{\mathfrak{D}^{n}_i}\longrightarrow & C^{n-1}(\hat{\Sigma}) \\
        f& \longmapsto & \mathfrak{D}^{n}_i(f)=f\circ s^{n-1}_i.
   \end{array}$$
The Alexander-Whitney coproduct dualizes to the cup product of cochains:
\begin{Pro}
  The singular cochain complex, $C^*(\hat{\Sigma})$ is naturally endowed by the cup product with
a differential graded algebra structure.
\end{Pro}

Concretely, the cup product is given, for $f\in C^p(\hat\Sigma)$, $g\in C^q(\hat\Sigma)$ and $\sigma\in \Sigma_{p+q}$ by
$$f\cup g(\sigma):=(f\otimes g)\circ \Delta_{p,q}(\sigma).$$
The unit of the cup product is the counit of $\Delta$:
the cochain in $C^0(\hat\Sigma)$ that maps each $0$-simplex to 1.

\section{Incidence algebras, relative Hochschild chains and cochains}\label{Hochschild}
In this section, we recall some classical definitions (incidence
algebras, relative Hochschild cochains) and extend them, featuring
in particular how the idea of {\it relativity} (to a separable
subalgebra) can be developped systematically to better account of
the properties of posets and their nerves. We state and prove an
homology comparison theorem and recall the classical CCT.
\begin{De} The incidence algebra $I_{\Sigma}$ of the poset $\Sigma$ is the associative unital
$\mathbb{K}$-algebra generated by the pairs of simplices $(\sigma,
\sigma')$ with $\sigma\le\sigma'$. The product of two pairs
$(\sigma, \sigma')$ and $(\beta, \beta')$ is $(\sigma, \beta')$ if
$\sigma'= \beta$ and 0 else. The unit of the algebra $1_I$ is the sum $\sum\limits_{\sigma\in\Sigma}(\sigma,\sigma)$.
\end{De}
The incidence algebra
$I_{\Sigma}$ has a (separable) commutative unital subalgebra $S_{\Sigma}$ generated as a
$\mathbb{K}$-algebra by the pairs $(\sigma, \sigma).$

\begin{Rem}\label{relincalg}
The category $\mathbf{BM}$ of $S_{\Sigma}$-bimodules (that is, left
and right modules over $S_\Sigma,$ where the right and left action
can be different) is a (nonsymmetric) tensor category for the tensor
product:
$$M\otimes_{BM} N:=M\otimes_{S_{\Sigma}}N,$$
where $M\otimes_{S_{\Sigma}}N$ is the quotient $S_{\Sigma}$-bimodule
of $M\otimes N$ by the relations $ma\otimes n - m\otimes an$ for
$m\in M, n\in N$ and $a\in S_{\Sigma}$. The same definition will
apply to tensor products of bimodules over an arbitrary ring.

The incidence algebra $I_{\Sigma}$ can then be better understood as an associative unital algebra in the tensor category  $\mathbf{BM}$. Notice in particular that the algebra product from $I_{\Sigma}\otimes I_{\Sigma}$ to $I_{\Sigma}$ factorises through the canonical projection from $I_{\Sigma}\otimes I_{\Sigma}$ to $I_{\Sigma}\otimes_{S_{\Sigma}} I_{\Sigma}$.
\end{Rem}

\begin{De} The barycentric incidence algebra $BI_{\Sigma}$ of the poset $\Sigma$ is the associative unital
$\mathbb{K}$-algebra generated by the $n$-simplices of $\hat\Sigma$, $(\sigma_0,\dots,\sigma_n)$, $n\geq 0$ with the product
$(\sigma_0,\dots,\sigma_n)\cdot (\sigma_{n+1},\dots,\sigma_{n+k})$ equal to $(\sigma_0,\dots,\sigma_n,\sigma_{n+2},\dots,\sigma_{n+k})$ if
$\sigma_n= \sigma_{n+1}$ and 0 else. The unit of the algebra is the sum $\sum\limits_{\sigma\in\Sigma}(\sigma)$.
\end{De}

\begin{Le}
The barycentric incidence algebra identifies with the tensor algebra $T_{S_\Sigma}(I_\Sigma):=\bigoplus\limits_{n\geq 0}I_\Sigma^{\otimes_{S_\Sigma}{n}}$ over $I_\Sigma$ in the category of $S_\Sigma$-bimodules, where we set $I_\Sigma^{\otimes_{S_\Sigma}{0}}:=S_\Sigma$.
\end{Le}

\begin{proof} Indeed, direct inspection shows that $I_\Sigma\otimes_{BM}I_\Sigma$ has as a basis the tensor products
 $(\sigma_0,\sigma_1)\otimes_{BM}(\sigma_1,\sigma_2)$ and more generally $I_\Sigma^{\otimes_{S_\Sigma}n}$ has as a
 basis the  tensor products $(\sigma_0,\sigma_1)\otimes_{BM}\dots \otimes_{BM}(\sigma_{n-1},\sigma_n)$. The isomorphism
 between $BI_\Sigma$ and the free associative algebra over $I_\Sigma$ in $\bf{BM}$ is then obtained as:
$$(\sigma_0)\longmapsto (\sigma_0,\sigma_0),$$
and for $n>0$,
$$(\sigma_0,\dots,\sigma_n)\longmapsto (\sigma_0,\sigma_1)\otimes_{BM}\cdots \otimes_{BM}(\sigma_{n-1},\sigma_n).$$

\end{proof}
\begin{Rem}
Notice that the map $$(\sigma_0,\sigma_1)\otimes_{BM}\cdots \otimes_{BM}(\sigma_{n-1},\sigma_n)\longrightarrow (\sigma_0,\sigma_n)$$ is a map of associative unital algebras from $T_{S_\Sigma}(I_\Sigma)$ to $I_\Sigma$.
\end{Rem}
\begin{Rem} Notice also that the arguments above show that there is a canonical retract from $I_\Sigma^{\otimes_{S_\Sigma}{n}}$ to $I_\Sigma^{\otimes n}$ given, in the natural basis of $I_\Sigma^{\otimes_{S_\Sigma}n}$ by:
$$ (\sigma_0,\sigma_1)\otimes_{BM}\cdots \otimes_{BM}(\sigma_{n-1},\sigma_n)\longmapsto  (\sigma_0,\sigma_1)\otimes\cdots \otimes(\sigma_{n-1},\sigma_n).$$
\end{Rem}

\begin{De}
The $S_\Sigma$-relative Hochschild chain complex of $I_\Sigma$ is the $S_\Sigma$ bimodule $T_{S_\Sigma}(I_\Sigma)$ equipped with
\begin{itemize}
\item a simplicial structure by the face maps
$$d_i^n(A)\longmapsto (\sigma_0,\sigma_1)\otimes_{BM}\cdots \otimes_{BM}(\sigma_{i-1},\sigma_{i+1})\otimes_{BM}\cdots\otimes_{BM}(\sigma_{n-1},\sigma_{n}),$$
for $0<i<n$, with
$$d_0^n(A)\longmapsto (\sigma_1,\sigma_2)\otimes_{BM}\cdots \otimes_{BM}(\sigma_{n-1},\sigma_{n}),$$
$$d_n^n(A)\longmapsto (\sigma_0,\sigma_1)\otimes_{BM}\cdots \otimes_{BM}(\sigma_{n-2},\sigma_{n-1}),$$
and the degeneracy maps
$$s_i^n(A)\longmapsto (\sigma_0,\sigma_1)\otimes_{BM}\cdots\otimes_{BM} (\sigma_{i-1},\sigma_i)\otimes_{BM} (\sigma_{i},\sigma_i) \otimes_{BM} (\sigma_{i},\sigma_{i+1})\cdots\otimes_{BM}(\sigma_{n},\sigma_{n+1}),$$
where $A:=(\sigma_0,\sigma_1)\otimes_{BM}\cdots \otimes_{BM}(\sigma_{n},\sigma_{n+1})$ and $i=0,\cdots,n$.
\item It is equipped with a differential coalgebra structure by the coproduct
$$I_\Sigma^{\otimes_{S_\Sigma}{n}}\hookrightarrow I_\Sigma^{\otimes{n}}\to \bigoplus\limits_{i=0}^nI_\Sigma^{\otimes{i}}\otimes I_\Sigma^{\otimes{n-i}}\to
\bigoplus\limits_{i=0}^nI_\Sigma^{\otimes_{S_\Sigma}{i}}\otimes I_\Sigma^{\otimes_{S_\Sigma}{n-i}},$$ where we use the retraction of
 $I_\Sigma^{\otimes_{S_\Sigma}{n+1}}$ into $I_\Sigma^{\otimes{n+1}}$, together with the differential induced by the simplicial structure,
  $\partial_n:=\sum\limits_{i=0}^n(-1)^{i}d_{i}^{n}$.
\end{itemize}
\end{De}

\begin{proof}
The linear isomorphism from $C_\ast(\hat\Sigma)$ to $T_{S_\Sigma}(I_\Sigma)$,
$$(\sigma_0,\cdots,\sigma_n)\longmapsto ((\sigma_0,\sigma_1)\otimes_{BM}\cdots \otimes_{BM}(\sigma_{n-1},\sigma_n))$$
induces (by structure transportation through the isomorphism) a simplicial and differential coalgebra structure on  $T_{S_\Sigma}(I_\Sigma)$. It is an easy exercise to check that it identifies with the structures given in the Definition.
\end{proof}

We get as a corollary an homological version of the CCT (the latter
to be stated below). Strangely enough, this Theorem does not seem to
have been observed and stated in the literature, at our best
knowledge, although allowing to directly connect the homology of
(finite) topological spaces to Hochschild homology.

\begin{Th}[Homology comparison theorem]
Given a finite topological space $\mathcal T$ with associated poset
$\Sigma$, there exists an isomorphism of simplicial vector spaces
$$C_\ast(\hat\Sigma)\cong T_{S_\Sigma}(I_\Sigma)$$
which induces an isomorphism in homology (with arbitrary coefficients)
$$H_\ast(\mathcal T)\cong H_\ast(\hat\Sigma)\cong H_\ast(T_{S_\Sigma}(I_\Sigma)).$$
\end{Th}

Let us turn now to the dual, cohomological, side of these questions.
\begin{De}[Relative Hochschild cochain complex]
The n-cochains
of the Hochschild cochain complex of $I_{\Sigma}$ relative to $S_{\Sigma}$ are
the elements of $\mbox{Hom}_{BM}(I_{\Sigma}^{\otimes_{S_{\Sigma}}n}, I_{\Sigma}).$ The relative Hochschild cochain complex is a differential algebra.
The coboundary map (the differential) is obtained (by right composition of morphisms) from the boundary map of the Hochschild chain complex.
The cup product is obtained (also by duality) from the coalgebra structure of the Hochschild chain complex.
\end{De}
Closed formulas will be given below for the differential.
Notice that this is the usual
formula for the Hochschild coboundary, extended to the relative setting.

Direct inspection shows
that $\mbox{Hom}_{BM}(I_{\Sigma}^{\otimes_{S_{\Sigma}}n}, I_{\Sigma}),$ is generated
linearly (over the ground field) by the maps sending a given tensor product
$((\sigma_0, \sigma_1)\otimes_{BM} (\sigma_1, \sigma_2)\otimes_{BM} \cdots\otimes_{BM} (\sigma_{n-1},
\sigma_{n}))$ to $(\sigma_{0}, \sigma_{n})$.

Observe also that the incidence algebra, $I_{\Sigma},$ is  a triangular algebra. The Hochschild cohomology of such algebras can be computed
explicitly by means of a spectral sequence, introduced by
S. Dourlens \cite{3}. We refer from now on to \cite{7} and \cite{3}
for the general properties of the Hochschild cohomology of
triangular and incidence algebras.

A key Theorem, due to Gerstenhaber and Schack \cite{6,7} shows the
key role of the relative Hochschild complex in relating the
cohomology of topological spaces with Hochschild cohomology.
Although little attention seems to have been paid to these results,
they appear to us as a key ingredient of the program of a
noncommutative geometry (although the latter program has been
developed historically in another direction).

\begin{Th}[Gerstenhaber-Schack]\label{comphh} The
relative Hochschild cochain complex, written $C^{*}_{S_{\Sigma}}(I_{\Sigma},
I_{\Sigma}),$ computes $HH^{\ast}(I_{\Sigma}, I_{\Sigma})$, the usual Hochschild cohomology of $I_{\Sigma}$.\end{Th}

\begin{Th}[Cohomology comparison theorem (CCT)] There is a cochain complex isomorphism $\iota$, that preserves the cup product, between the
singular cochain complex of $\hat{\Sigma}$, $C^{\ast}(\hat{\Sigma})$  and
the relative Hochschild cochain complex of the incidence algebra $I_{\Sigma},$
$$C^{*}_{S_{\Sigma}}(I_{\Sigma}, I_{\Sigma}):=
\{Hom_{BM}(I_{\Sigma}^{\otimes_{S_{\Sigma}} n},I_{\Sigma})\}_{n\ge
1}.$$ The isomorphism is given by:
$$\iota(f)((\sigma_0, \sigma_1)\otimes_{BM} (\sigma_1, \sigma_2)\otimes_{BM}\cdots\otimes_{BM}
(\sigma_{n-1}, \sigma_{n})) := f(\sigma_0\leq \sigma_1 \sigma\leq
\cdots \leq \sigma_n)\cdot (\sigma_0, \sigma_n).$$ for $f \in C^n(\hat{\Sigma}_{n})$. In particular, by Thm \ref{comphh},
$$HH^{\ast}(I_{\Sigma}, I_{\Sigma}) \cong H^{\ast}(\hat{\Sigma},
\mathbb{K}).$$
\end{Th}
 \begin{proof} A proof will follow from the finer cosimplicial comparison theorem to be stated in the next section of the article. The original proof of the Theorem (in this form, as there are various variants of the CCT in Gerstenhaber and Schack's works) can be found in \cite[Thm 138, Section 15]{7}. See the same article
   for details,
generalizations, applications (in geometry and deformation theory) and a survey of the history of this theorem and its various variants.
\end{proof}

\section{The cosimplicial comparison theorem}\label{cosimpl}
In this section, we prove that the CCT can be enhanced to the cosimplicial level. For completeness sake we give some details on the proof, as the result also implies the classical CCT.
Recall that the maps $d_i^n,\ s_i^n$ endow $T_{S_\Sigma}(I_\Sigma)$ with a simplicial structure. By duality, this yields a cosimplicial structure
 on the relative Hochschild cochain complex $C^{*}_{S_{\Sigma}}(I_{\Sigma},
I_{\Sigma})$ whose coface and codegeneracy maps are :
    \begin{center}
        $F^{n}_i:Hom_{BM}(I_{\Sigma}^{\otimes_{S_{\Sigma}}
n-1},I_{\Sigma})\longrightarrow
Hom_{BM}(I_{\Sigma}^{\otimes_{S_{\Sigma}}
n},I_{\Sigma})$\\
        $f\longmapsto F^{n}_i(f)= f\circ d^{n}_i$
    \end{center}
    \begin{center}
        $D_i^{n}:Hom_{BM}(I_{\Sigma}^{\otimes_{S_{\Sigma}}
n},I_{\Sigma})\longrightarrow
Hom_{BM}(I_{\Sigma}^{\otimes_{S_{\Sigma}}
n-1},I_{\Sigma})$\\
        $f\longmapsto D_i^{n}(f)=f\circ ^Is^{n}_i.$
    \end{center}
Let us explicitely check one of the cosimplicial identities. For  $0\leq i <j\leq n$ and $f\in Hom_{S_{\Sigma}}(I_{\Sigma}^{\otimes_{S_{\Sigma}}
n},I_{\Sigma})$, we have
    \begin{align*}
        D_i^{n-1} D_j^n(f)&=D_i^{n-1}(f\circ s_j^n)=(f\circ s_j^n)\circ s_i^{n-1};\\
        &=f\circ(s_j^n\circ s_i^{n-1})=f\circ(s_i^{n}\circ s_{j-1}^{n-1});\\
        &=(f\circ s_i^n)\circ s_{j-1}^{n-1}=D_{j-1}^{n-1}(f\circ s_i^n);\\
        &=D_{j-1}^{n-1}D_i^n(f).
    \end{align*}\\

\begin{Th}[Cosimplicial comparison theorem]
 The cochain complexes isomorphism $\iota:C^{\ast}(\hat{\Sigma})\longrightarrow C^{*}_{S_{\Sigma}}(I_{\Sigma}, I_{\Sigma})$
preserves the cosimplicial structures.
\end{Th}

\begin{proof} Let us explicitely check for example that the map $\iota$ is compatible with the degeneracy operators, that is
    that the diagrams below are commutative\\
    \xymatrix{ C^n(\hat{\Sigma})\ar[dd]_{\iota_n} \ar[r]^{\mathfrak{D}^{n}_i}& C^{n-1}(\hat{\Sigma}) \ar[dd]^{\iota_{n-1}}&\\
    && i.e., D^{n}_i\circ \iota_n= \iota_{n-1}\circ \mathfrak{D}^{n}_i \qquad (I)\\
        C^{n}_{S_{\Sigma}}(I_{\Sigma},I_{\Sigma}) \ar[r]^{D^{n}_i}& C_{S_{\Sigma}}^{n-1}(I_{\Sigma},I_{\Sigma})
        &}\\
        Indeed for any $f\in C^n(\hat{\Sigma})$, one has:
        \begin{align*}
        (D_i^{n}\circ \iota_n(f))((\sigma_0,\sigma_1)\otimes_{S_{\Sigma}}\cdots\otimes_{S_{\Sigma}}(\sigma_{n-2},\sigma_{n-1}))
        &=D_i^{n}(\iota_n(f))((\sigma_0,\sigma_1)\otimes_{S_{\Sigma}}\cdots\otimes_{S_{\Sigma}}(\sigma_{n-2},\sigma_{n-1}))\\
        &=(\iota_n(f)\circ d_i^{n})((\sigma_0,\sigma_1)\otimes_{S_{\Sigma}}\cdots\otimes_{S_{\Sigma}}(\sigma_{n-2},\sigma_{n-1}))\\
        &=\iota_n(f)((\sigma_0,\sigma_1)\otimes_{S_{\Sigma}}\cdots\otimes_{S_{\Sigma}}(\sigma_{i-1},\sigma_i)\otimes_{S_{\Sigma}}(\sigma_i,\sigma_i)\otimes_{S_{\Sigma}} \\
        &(\sigma_i,\sigma_{i+1})\otimes_{S_{\Sigma}}\cdots\otimes_{S_{\Sigma}}(\sigma_{n-2},\sigma_{n-1}))\\
        &=f(\sigma_0\le\sigma_1\le\cdots\le\sigma_{i-1}\le\sigma_i\le\sigma_i\le\sigma_{i+1}\le\cdots\le\sigma_{n-1}) \\
        &(\sigma_0,\sigma_{n-1})\\
        &=(f\circ {d}^{n}_i)(\sigma_0\le\sigma_1\le\cdots\le\sigma_{i-1}\le\sigma_i\le\sigma_{i+1}\le\cdots\le\sigma_{n-1}) \\
        &(\sigma_0,\sigma_{n-1})\\
        &=\iota_{n-1}(f\circ {d}^{n}_i)((\sigma_0,\sigma_1)\otimes_{S_{\Sigma}}\cdots\otimes_{S_{\Sigma}}(\sigma_{n-2},\sigma_{n-1}))\\
        &=(\iota_{n-1}\circ \mathfrak{D}^{n}_i)(f)((\sigma_0,\sigma_1)\otimes_{S_{\Sigma}}\cdots\otimes_{S_{\Sigma}}(\sigma_{n-2},\sigma_{n-1})).
    \end{align*}
\end{proof}

\section{Operadic comparison theorem}\label{opera}

\begin{De}\label{nsop}
   A nonsymmetric operad (or operad for short, in this article) over the category of
    $\mathbb{K}$-vector spaces is a collection of vector spaces $\{\mathcal{O}(k)\mid k\ge 1\}$ together with
    composition products:
    $$\begin{array}{ccc}\mathcal{O}(k)\otimes\mathcal{O}(n_1)\otimes\cdots\otimes\mathcal{O}(n_k)&\stackrel
    {\gamma^{\mathcal{O}}}\longrightarrow & \mathcal{O}(n_1+\cdots+n_k)\\
        x\otimes x_1\otimes\cdots\otimes x_k & \longmapsto & \gamma^{\mathcal{O}}(x;x_1,\cdots,x_k)
    \end{array}$$
    which are:
    \begin{enumerate}
\item associative in the sense that
        \begin{align*}
    \gamma^{\mathcal{O}}(\gamma^{\mathcal{O}}(x;x_1,\cdots,x_k);y_1,\cdots,y_{n_1+\cdots+n_k})&=\gamma^{\mathcal{O}}(x;\gamma^{\mathcal{O}}(x_1;y_1,\cdots,y_{n_1}),\gamma^{\mathcal{O}}(x_2;y_{n_1+1},\cdots,y_{n_1+n_2})\\
    &\cdots,\gamma^{\mathcal{O}}(x_k;y_{n_1+\cdots+n_{k-1}+1},\cdots,y_{n_1+\cdots+n_{k-1}+n_k})),
        \end{align*}
 \item there is an identity element $1_{\mathcal{O}}\in\mathcal{O}(1)$, also called simply the unit of the operad, such that
       $$\gamma^{\mathcal{O}}(x;\underbrace{1_{\mathcal{O}},\cdots,1_{\mathcal{O}}}_{\mbox{k times}})=x =\gamma^{\mathcal{O}}(1_{\mathcal{O}};x).$$
    \end{enumerate}
\end{De}

\begin{De}
{\it Let $\mathcal{O}$ and $\mathcal{O'}$ be two $\Sigma$-operads
with respective composition products $\gamma^{\mathcal{O}}$ and
$\gamma^{\mathcal{O}^{\prime}}$ and respective associated units
$1_{\mathcal{O}} \hspace{2mm}\mbox{and}\hspace{2mm}
1_{\mathcal{O'}}.$ A morphism of operads
$\mathcal{O}\stackrel{f}\longrightarrow\mathcal{O'}$ is a collection
$\{f_n:\mathcal{O}(n)\longrightarrow\mathcal{O'}(n)\}_{n\geq 0}$ of
vector space morphisms such that:
\begin{enumerate}
\item $f(1_{\mathcal{O}})= 1_{\mathcal{O'}}$;
\item $ f_j(\gamma^{\mathcal{O}}(x_0\otimes x_1 \otimes...\otimes x_n))= \gamma^{\mathcal{O}^{\prime}}(f_n(x_0)\otimes f_{i_1}(x_1) \otimes...\otimes f_{i_n}(x_n))$ with $j=i_1 +i_2 +...+i_n.$
\end{enumerate}}
\end{De}

\begin{Rem}
\begin{enumerate}
\item Equivalently an operad can also be defined by the so called partial
compositions:
$$\begin{array}{ccc}
 \mathcal{O}(m) \otimes \mathcal{O}(n)& \stackrel{\circ_{i}}\longrightarrow & \mathcal{O}(m + n - 1)\qquad , 1\leq i\leq m\\
x \otimes y  & \longmapsto & x \circ_{i} y\\
\end{array}$$
satisfying some properties inherited from Definition \ref{nsop} (see e.g. \cite{dotsenko,lv} for explicit axioms).\\
    The two definitions are related as follows:
\begin{enumerate}
\item[(1-i)] $x \circ_{i} y = \gamma^{\mathcal{O}}(x; \overbrace{id, \cdots,\underbrace{y}_{i}, \cdots  id}^{m-tuple} );  \qquad  1\leq i\leq m.$
\item[(1-ii)] $\gamma^{\mathcal{O}}(x;  y_{1},\cdots , y_{m})= (\cdots(((x \circ_{m} y_{m}) \circ_{m-1} y_{m-1})\cdots ) \circ_{1} y_{1}$
\end{enumerate}
\item An operad, $\mathcal{O},$ is said to be \textbf{multiplicative} if there exists $m\in\mathcal{O}(2)$
such that $m\circ_1m=m\circ_2m$.
\item An operad $\mathcal{O}$ is said to be \textbf{unitary} if the unit morphism $\eta:\mathbb{K}\longrightarrow\mathcal{O}(1)$
is an isomorphism in which case the unit element of $\mathcal{O}$
will be denoted: $1_{1}= 1_{\mathcal{O}}= \eta(1_{\mathbb{K}}) \in
\mathcal{O}(1).$
\end{enumerate}
\end{Rem}
\begin{Ex}
The fundamental example of an operad is the operad $\mathcal{L}_V$
of multilinear endomorphisms of a vector space $V$, called
endomorphism operad and defined by: for all $n\ge 1$,
$\mathcal{L}_V(n)=Hom_{\mathbb{K}}(V^{\otimes n},V)$. The
composition product is obtained from the composition of multilinear
maps. The associated unit is the identity map:
$V\stackrel{id_V}\longrightarrow V$.\end{Ex}
\begin{Ex}
The fundamental example of a multiplicative operad is the operad
$\mathcal{L}_A$ of multilinear endomorphisms of an associative algebra $A$: for all
$n\ge 1$, we have again $\mathcal{L}_A(n)=Hom_{\mathbb{K}}(A^{\otimes n},A)$. The
associated unit is the identity map, the multiplication $m$ is the algebra product.

Multiplication is essential for our forthcoming developments. For example, the existence of the product $m$ is what allows the definition of a differential on the Hochschild cochain complex of an associative algebra, or the definition of the cup product.\end{Ex}

\begin{Th}
Let $S_\Sigma -Hom(I_\Sigma^{\otimes n},I_\Sigma)$ to be the $S_\Sigma$-bimodule of $S_\Sigma$-bimodule morphisms from $I_\Sigma^{\otimes n}$ to $I_\Sigma$ that factorize through $I_\Sigma^{\otimes_{S_\Sigma}n}$. We set:
$$End_{I_\Sigma,S_\Sigma}(n):=S_\Sigma -Hom(I_\Sigma^{\otimes n},I_\Sigma).$$
The family $(End_{I_\Sigma,S_\Sigma}(n))_{n\geq 1}$ is a suboperad of the endomorphism operad $\mathcal L_{I_\Sigma}$.
\end{Th}

In the Theorem, the $S_\Sigma$-bimodule structure of $I_\Sigma^{\otimes n}$ is given by:
$$(\beta,\beta)\cdot ((\sigma_0,\sigma_1)\otimes\dots\otimes(\sigma_{n-1},\sigma_n))\cdot (\gamma,\gamma):=\partial_{\beta,\sigma_0}\partial_{\sigma_n,\gamma}(\sigma_0,\sigma_1)\otimes\dots\otimes(\sigma_{n-1},\sigma_n),$$
and the factorization property means that a $\mu\in Hom(I_\Sigma^{\otimes n},I_\Sigma)$ is required to factorize as:
$$\mu:I_\Sigma^{\otimes n}\to I_\Sigma^{\otimes_{S_\Sigma} n}\to I_\Sigma.$$
\begin{proof}
The space $End_{I_\Sigma,S_\Sigma}(n)$ is spanned by the maps:
$$\mu_{(\sigma_0,\cdots,\sigma_n)}:((\beta_0,\beta_1)\otimes (\beta'_1,\beta_2)\otimes\cdots\otimes (\beta'_{n-1},\beta_n))\longmapsto \partial_{\beta_1,\beta'_1}\cdots\partial_{\beta_{n-1},\beta'_{n-1}}\partial_{(\beta_0,\cdots,\beta_n),(\sigma_0,\cdots,\sigma_n)}\cdot (\sigma_0,\sigma_n).$$

The first component $End_{I_\Sigma,S_\Sigma}(1)$ contains in particular $\sum\limits_{\sigma_0\leq \sigma_1\in\Sigma}\mu_{(\sigma_0,\sigma_1)}$, which is the identity of $I_\Sigma$. Moreover, given $\mu_{(\sigma_0,\cdots,\sigma_n)}$ and $\mu_{(\beta_0,\cdots,\beta_k)}$, we have
\begin{equation}\label{eqqun}\mu_{(\sigma_0,\cdots,\sigma_n)}\circ_i\mu_{(\beta_0,\cdots,\beta_k)}=\partial_{\sigma_{i-1},\beta_0}\partial_{\sigma_i,\beta_k} \mu_{(\sigma_0,\cdots,\sigma_{i-1}=\beta_0,\beta_1,\cdots,\beta_k=\sigma_i,\sigma_{i+1},\cdots,\sigma_n)},
\end{equation}
where $\circ_i$ stands for the composition product in $\mathcal L_{I_\Sigma}$,
which concludes the proof.
\end{proof}
\begin{Ex}[Cochain operads]
The simplicial cochain complex $C^\ast(\hat\Sigma)$ has an operad structure defined by \cite{8}:
$$\gamma: C^k(\hat\Sigma)\otimes C^{n_1}(\hat\Sigma)\otimes \cdots\otimes C^{n_k}(\hat\Sigma)\to C^{n_1+\cdots+n_k}(\hat\Sigma)$$
$$\gamma((\sigma_0,\cdots,\sigma_k)^\ast\otimes(\beta_0^1,\cdots,\beta_{n_1}^1)^\ast\otimes\cdots\otimes(\beta_1^k,\cdots,\beta_{n_k}^k)^\ast)(\gamma_0,\cdots,\gamma_{n_1+\cdots+n_k})$$
$$:=\partial_{(\sigma_0,\cdots,\sigma_k),(\gamma_0,\gamma_{n_1},\cdots,\gamma_{n_1+\cdots+n_k})}\partial_{(\gamma_0,\cdots,\gamma_{n_1}),(\beta_0^1,\cdots,\beta_{n_1}^1)}\cdots\partial_{(\gamma_{n_1+\cdots+n_{k-1}},\cdots,\gamma_{n_1+\cdots+n_k}),(\beta_1^k,\cdots,\beta_{n_k}^k)},$$
where $(\sigma_0,\cdots,\sigma_k)^\ast$ denotes the $k$-cochain whose value on $(\beta_0,\cdots,\beta_k)$ is $\partial_{(\sigma_0,\cdots,\sigma_k),(\beta_0,\cdots,\beta_k)}$. The unit of the operad is $\sum\limits_{\sigma_0\leq\sigma_1}(\sigma_0,\sigma_1)^\ast$.
\end{Ex}
\begin{Th}[Operadic comparison theorem] The map
$$(\sigma_0,\cdots,\sigma_k)^\ast\longmapsto \mu_{(\sigma_0,\cdots,\sigma_k)}$$
induces an isomorphism of operads
$$C^\ast(\hat\Sigma)\cong End_{I_\Sigma,S_\Sigma}(\ast).$$
\end{Th}
\begin{proof}
We indeed have
$$(\sigma_0,\cdots,\sigma_n)^\ast\circ_i(\beta_0,\cdots,\beta_k)^\ast=$$
$$\gamma((\sigma_0,\cdots,\sigma_n)^\ast\otimes (\sum\limits_{\gamma_0\leq \gamma_1\in\Sigma}(\gamma_0,\gamma_1)^\ast \otimes\cdots\otimes\sum\limits_{\gamma_0\leq \gamma_1\in\Sigma}(\gamma_0,\gamma_1)\otimes (\beta_0,\cdots,\beta_{n_i})^\ast\otimes$$
$$\sum\limits_{\gamma_0\leq \gamma_1\in\Sigma}(\gamma_0,\gamma_1)^\ast\otimes\cdots\otimes \sum\limits_{\gamma_0\leq \gamma_1\in\Sigma}(\gamma_0,\gamma_1)^\ast)$$
$$=\partial_{\sigma_{i-1},\beta_0}\partial_{\sigma_{i},\beta_{k}}(\sigma_0,\cdots,\sigma_{i-1},\beta_1,\cdots,\beta_{n_i-1},\sigma_{i},\cdots,\sigma_k)^\ast ,$$
that agrees with Eq. (\ref{eqqun}).
\end{proof}
\begin{Co}
There is a canonical embedding of the cochain operad of a poset into the endomorphism operad of its incidence algebra:
$$C^\ast(\hat\Sigma)\hookrightarrow \{Hom(I_\Sigma^{\otimes \ast},I_\Sigma)\}.$$
\end{Co}
\section{Relative operads}\label{relop}
The previous operadic framework can be naturally refined to the relative setting that, we claim, is the right one to study poset combinatorics and its links to topology and homological algebra.
We let $B$ be an associative unital algebra.
\begin{De}
A $B$-relative operad (relative operad, for short) is a collection of $B$-bimodules
 $\{\mathcal{O}(k)\mid k\ge 1\}$ together with
    $B$-bimodule morphisms:
    $$\begin{array}{ccc}{\gamma^{\mathcal{O}}}:\mathcal{O}(k)&\to&Hom_{B-BiMod}(\mathcal{O}(n_1)\otimes_B\cdots\otimes_B\mathcal{O}(n_k),
    \mathcal{O}(n_1+\cdots+n_k))\\
        x & \longmapsto & (x_1\otimes_B\cdots\otimes_B x_k\longmapsto\gamma^{\mathcal{O}}(x;x_1,\cdots,x_k))
    \end{array}$$
    which are:
    \begin{enumerate}
\item associative in the sense that
        \begin{align*}
    \gamma^{\mathcal{O}}(\gamma^{\mathcal{O}}(x;x_1,\cdots,x_k);&y_1,\cdots,y_{n_1+\cdots+n_k})=\\
    &\gamma^{\mathcal{O}}(x;\gamma^{\mathcal{O}}(x_1;y_1,\cdots,y_{n_1}),\gamma^{\mathcal{O}}(x_2;y_{n_1+1},\cdots,y_{n_1+n_2}),\cdots\\
   & \cdots,\gamma^{\mathcal{O}}(x_k;y_{n_1+\cdots+n_{k-1}+1},\cdots,y_{n_1+\cdots+n_{k-1}+n_k})),
        \end{align*}
 \item there is an identity element $1_{\mathcal{O}}\in\mathcal{O}(1)$, also called simply the unit of the operad, such that
       $$\gamma^{\mathcal{O}}(x;\underbrace{1_{\mathcal{O}},\cdots,1_{\mathcal{O}}}_{\mbox{k times}})=x =\gamma^{\mathcal{O}}(1_{\mathcal{O}};x).$$
    \end{enumerate}
\end{De}
\begin{Ex}[Relative endomorphism operad] Let $A$ be an associative unital algebra and $B$ a subalgebra. The $B$-relative endomorphism operad of $A$ is the relative operad defined by:
$$\{Hom_{B-BiMod}(A^{\otimes_B n},A),\ n\geq 1\}$$
with structure operator $\gamma$ defined by the composition of multilinear $B$-bimodules morphisms.
\end{Ex}
\begin{Ex}[Operator-valued probability operad] A particularly interesting example of relative operad originates in noncommutative probability that aims at encoding a notion of conditional probability suited for operator-valued random variables, see e.g. \cite{Mingo}.

In that case, we consider
$$\{Hom_{B-BiMod}(A^{\otimes_B n},B)|\ n\geq 1\}.$$
It is obviously a suboperad of the relative endomorphism operad of $A$: the operadic composition of $B$-valued morphisms is a $B$-valued morphism.
\end{Ex}
\begin{Ex}
The most meaningful example for our purposes originates in the relative Hochschild complex.  We consider here:
$$\{Hom_{BM}(I_\Sigma^{\otimes_{S_\Sigma}n},I_\Sigma)|{n\geq 1}\}.$$
The previous computations in the article show that the structure map defining the operadic structure on $\{S_\Sigma-Hom(I_\Sigma^{\otimes n},I_\Sigma)\}_{n\geq 1}$ go over and also define an operadic structure on $$\{Hom_{BM}(I_\Sigma^{\otimes_{S_\Sigma}n},I_\Sigma)|{n\geq 1}\}.$$

Using the CCT isomorphism $\iota$, the same construction can also be performed at the cochain algebra level (we leave the exercise to the reader).
\end{Ex}

\section{Brace differential graded algebras}\label{bdga}
Let us introduce now BDGAs. These algebras first appeared in the
work of Getzler-Jones on algebras up to homotopy (without a
specific name) as a particular case of $B_\infty$-algebras, associated
in particular to Hochschild complexes of associative algebras, see
\cite[Sect. 5.2]{gj}. When Gerstenhaber and Voronov studied them more
in detail \cite{8,vg,v}, they decided to call these algebras homotopy
$G$-algebras. However, this
terminology appeared to be a misleading one after Tamarkin had shown
that the name $G$(erstenhaber)-algebra up to homotopy should be
naturally given to another class of algebras \cite{t}. We
call them by a name that reflects their properties and
should not create confusion, namely: brace differential graded
algebras.\par The basic idea is that BDGAs are
associative differential graded algebras together with extra (brace)
operations that behave exactly as the Kadeishvili-Getzler brace
operations on the Hochschild cohomological complex of an associative
algebra \cite{k,ge}. We write, as usual, $B(A)$ for the cobar coalgebra
over a differential graded algebra (DGA) $A$, where the product is
written $\cdot$  and the differential (of degree +1) $\d$. That is,
$B(A)$ is the cofree graded coalgebra
$T(A[1]):=\bigoplus\limits_{n\in\NM}A[1]^{\ot n}$ over the desuspension
$A[1]$ of $A$ ($A[1]_n:=A_{n+1}$). We use the bar notation and write
$[a_1|...|a_n]$ for $a_1\ot ...\ot a_n\in A[1]^{\ot n}$. In particular,
the coproduct on $T(A[1])$ is given by: $$\Delta
[a_1|...|a_n]:=\sum\limits_{i=0}^n[a_1|...|a_i]\ot [a_{i+1}|...|a_n].$$
There is a differential coalgebra structure on $B(A)$ induced by the
DGA structure on $A$. In fact, since $B(A)$ is cofree as a graded
coalgebra, the properties of the cofree coalgebra functor imply that,
in general, a coderivation $D\in Coder(B(A))$ is entirely determined by
the composition (written as a degree 0 morphism): $$\tilde D:\
B(A)\mapright{D}B(A)[1]\mapright{p}A[2],$$ where $p$ is the natural
projection. In particular, the differential $d$ on $B(A)$ is induced by
the maps: $$\d :A[1]\mapright{} A[2],$$ and $$\mu :A[1]\ot
A[1]\mapright{} A[2],$$ where $\mu (a,b):=(-1)^{|a|}a\cdot b$. The
algebra $A$ is a BDGA if it is provided with a set of
extra-operations called the braces: $$B_k: A[1]\ot A[1]^{\ot
k}\mapright{} A[1],\ k\geq 1,$$ satisfying certain relations. These
relations express exactly the fact that the braces have to induce a
differential Hopf algebra structure on $B(A)$. Explicitly, the
relations satisfied by the braces are then \cite[Sect. 5.2]{gj} and
\cite{kh,v} (we use Getzler's notation:
$v\{v_1,...,v_n\}:=B_n(v\ot (v_1\ot
...\ot v_n))$):
\begin{enumerate}
\item The brace relations
(the associativity relations for the product on $B(A)$).
$$(v\{v_1,...,v_m\})\{w_1,...,w_n\}=\sum\limits_{0\leq i_1\leq
j_1\leq ...\leq i_m\leq j_m\leq
n}(-1)^{\sum\limits_{k=1}^m(|v_k|-1)(\sum\limits_{l=1}^{i_k}(|w_l|-1))}$$
$$v\{w_1,...,w_{i_1},v_1\{w_{i_1+1},...,w_{j_1}\},w_{j_1+1},...,
v_m\{w_{i_m+1},...,w_{j_m}\},w_{j_m+1},...,w_n\},$$ with the usual
conventions on indices: for example, an expression such as
$v_5\{w_7,...,w_6\}$ has to be read $v_5\{\emptyset\}=v_5$. \item The
distributivity relations of the product w.r. to the braces. $$(v\cdot
w)\{ v_1,...,v_n\}=\sum\limits_{k=0}^n(-1)^{|
w|\sum\limits_{p=1}^k(|v_p| -1)}v\{v_1,...,v_k\}\cdot
w\{v_{k+1},...,v_n\},$$
\item The boundary relations.
$$\d (v\{v_1,...,v_n\})-\d v\{v_1,...,v_n\}$$
$$+\sum\limits_{i=1}^n(-1)^{|v|+|v_1|+...+|v_{i-1}|-i+1}v\{v
_1,...,\d v_i,...,v_n\}$$ $$=(-1)^{|v|(|v_1|-1)}v_1\cdot (
v\{v_2,...,v_n\})$$
$$-\sum\limits_{i=1}^{n-1}(-1)^{|v|+|v_1|+... + | v_{ i
}|-i-1}v\{v_1,...,v_i\cdot v_{i+1},...,v_n\}$$
$$+(-1)^{|v|+|v_1|+...+| v_{ n-1 }|-n}(v\{v_1,...,v_{n-1}\})\cdot
v_n.$$

\end{enumerate}

\begin{Ex}
There is a  BDGA structure on
the Hochschild cochain complex $C^{\ast}(A,A)$ of an associative
algebra $A$ over a commutative unital ring $k$ \cite{8}. Recall that
$C^n(A,A)=Hom_k(A^{\ot n},A)$. The brace operations on
$C^{\ast}(A,A)$ are the multilinear operators defined for
$x,x_1,...,x_n$ homogeneous elements in $C^{\ast}(A,A)$ and
$a_1,...,a_m$ elements of $A$ by:\\
$$\{x\}\{x_1,...,x_n\}(a_1,...,a_m):=\sum\limits_{0\leq i_1\leq
i_1+|x_1|\leq i_2\leq ...\leq i_n+|x_n|\leq
n}(-1)^{\sum\limits_{k=1}^ni_k\cdot (|x_k|-1)}$$
$$x(a_1,...,a_{i_1},x_1(a_{i_1+1},...,a_{i_1+|x_1|}),...,a_{i_n},x_n(a_
{ i _n+1},...,a_{i_n+|x_n|}),...a_m).$$
The other operations defining the BDGA structure, $\d$
and $\cdot$ are, respectively, the
Hochschild coboundary and the cup product.\end{Ex}

\begin{Pro}The canonical embedding of $S_\Sigma-Hom(I_\Sigma^{\otimes n},I_\Sigma)\cong Hom_{BM}(I_\Sigma^{\otimes_{S_\Sigma}n},I_\Sigma)$ into $Hom(I_\Sigma^{\otimes n},I_\Sigma)$ induces a  BDGA structure on
the relative Hochschild cochain complex $\{Hom_{BM}(I_\Sigma^{\otimes_{S_\Sigma}n},I_\Sigma)\}_{n\geq 1}$.\end{Pro}

The Proposition follows from the observation that $S_\Sigma$-equivariance properties are preserved by the brace operations, that are obtained by iterated compositions of $S_\Sigma$-equivariant morphisms.
\begin{Ex}
There is
a BDGA structure on the
cochain complex of a simplicial set \cite{gj,8}.
Recall that a simplicial set is a contravariant functor
from the category $\bf\Delta$ of finite sets $[n]=\{0,...,n\}$ and
increasing morphisms to $\bf Set$. For a simplicial set $S:{\bf
\Delta}\lra \bf Set$, for $\sigma\in S_n:=S([n])$, and for a strictly
increasing sequence $0\leq a_0< ...<a_m\leq n$, we write $\sigma
(a_0,...,a_m)$ for $S(i_a) (\sigma )\in S_m$, where $i_a$ is the
unique map from $[m]$ to $[n]$ sending $[m]$ to $\{a_0,...,a_m\}$.
Define a map $\Delta_{1,r}$ from the singular complex of $\Sigma$, $C_\ast
(\hat\Sigma)$ to $C_\ast (\hat\Sigma)\ot C_\ast (\hat\Sigma)^{\ot r}$ as follows. For $\sigma\in \hat\Sigma_n$, set:\\
$$\Delta_{1,r}(\sigma ):=\sum\limits_{0\leq b_1'\leq b_1\leq ...\leq
b_r'\leq b_r\leq n}(-1)^{\sum\limits_{k=1}^r((b_k-b_k')b_k')}$$
$$\sigma
(0,1,...,b_1',b_1,b_1+1,...,b_2',b_2,...,b_r',b_r,...,n-1,n)$$ $$\ot
(\sigma (b_1',...,b_1)\ot \sigma (b_2',...,b_2)\ot ...\ot \sigma
(b_r',...,b_r)).$$
Dualizing $\Delta_{1,r}$, we get a
map from $C^\ast (\hat\Sigma)\ot C^\ast (\hat\Sigma)^{\ot r}$ to $C^\ast
(\hat\Sigma)$. By analogy with the case of Hochschild cochains, we write
$\s\{\s_1,...,\s_r\}$ for $\Delta_{1,r}^\ast (\s\ot (\s_1\ot ...\ot
\s_r))$. These brace operations on cochains, together with the
simplicial coboundary and the cup product induce a BDGA
structure on the bar coalgebra on $C^\ast (\hat\Sigma)$.\end{Ex}

\begin{Pro}The isomorphism
$\iota$ commutes with the action of the brace operations on $C^\ast
(\hat\Sigma)$ and $C_{S_\Sigma}^\ast (I_\Sigma ,I_\Sigma
)$.\end{Pro}
\begin{proof}
Indeed, let $f,f_1,...,f_k$ belong respectively to
$C^n(\hat\Sigma )$, $C^{n_1}(\hat\Sigma )$,..., $C^{n_k}(\hat\Sigma )$.
Let
$(\sigma_0\leq
\sigma_1\leq \sigma_2\leq ...\leq \sigma_{m-1}\leq \sigma_m)\in
\hat\Sigma_m$, where $m:=n+n_1+...+n_k-k $. Let
us also introduce the following useful convention. Let e.g.
$(\sigma_{i_0},\sigma_{i_1},k_1,k_2,\sigma_{i_3},...,\sigma_{i_q},k_p)$
be any sequence, the elements of which are either scalars (the $k_i$s), either
simplices of $\Sigma$ (the $\sigma_i$s), and assume that
$(\sigma_{i_0}\leq\sigma_{i_1}\leq ...\leq\sigma_{i_q})$ is a simplex
of $\hat\Sigma$. Then, we write
$f(\sigma_{i_0},\sigma_{i_1},k_1,k_2,\sigma_{i_3},...,\sigma_{i_q},k_p)
$ for $(\prod_{i=1}^pk_i)\cdot f(\sigma_{i_0}\leq\sigma_{i_1}\leq
...\leq\sigma_{i_q})$.\par We have, according to the definition
of the braces (we omit the signs for lisibility, following a standard practice in algebraic topology):\par $f\{f_1,...,f_k\}(\sigma_0\leq\sigma_1\leq
...\leq\sigma_m)$ $$=\sum\pm
f(\sigma_0,...,\sigma_{i_1},f_1(\sigma_{i_1}\leq
...\leq\s_{i_1+n_1}),\s_{i_1+n_1},...,\s_{i_k},$$
$$f_k(\sigma_{i_k}\leq ...\leq\s_{i_k+n_k}),\s_{i_k+n_k},...,\s_m).$$
Therefore:\par $\iota
(f\{f_1,...,f_k\})((\sigma_0,\sigma_1),(\sigma_1,\sigma_2),...,(\sigma_
{m-1},\s_m))$ $$=\{\sum\pm
f(\sigma_0,...,\sigma_{i_1},f_1(\sigma_{i_1}\leq
...\leq\s_{i_1+n_1}),\s_{i_1+n_1},...,\s_{i_k},$$
$$f_k(\sigma_{i_k}\leq
...\leq\s_{i_k+n_k}),\s_{i_k+n_k},...,\s_m)\}\cdot (\s_0,\s_m)$$
$$=\sum\pm \iota
(f)((\sigma_0,\sigma_1),...,(\sigma_{i_1-1},\sigma_{i_1}),f_1(\sigma_{i
_1}\leq ...\leq\s_{i_1+n_1})\cdot (\sigma_{i_1},\sigma_{i_1+n_1}),$$
$$(\sigma_{i_1+n_1},\sigma_{i_1+n_1+1}),...
,(\sigma_{i_k-1},\sigma_{i_k}),$$ $$f_k(\sigma_{i_k}\leq
...\leq\s_{i_k+n_k})\cdot
(\sigma_{i_k},\sigma_{i_k+n_k}),...,(\s_{m-1},\s_m))$$ $$=\sum\pm \iota
(f)\{\iota (f_1),...,\iota
(f_k)\}((\sigma_0,\sigma_1),(\sigma_1,\sigma_2),...,(\sigma_{m-1},\s_m)
),$$ and the proof of the proposition follows.\end{proof}

\begin{Th}The morphism $\iota$ is an isomorphism of
BDGAs between the singular
cochain complex of the barycentric subdivision of a finite simplicial
complex $\Sigma$ and the $S_\Sigma$-relative Hochschild cochain complex
of the incidence algebra of $\Sigma$.\par In particular, as the embedding of the latter complex into the classical Hochschild cochain complex
of the incidence algebra of $\Sigma$ is also a morphism of BDGAs, besides being a quasi-isomorphism,
the cohomology
comparison theorem of Gerstenhaber and Schack relating singular
cohomology and Hochschild cohomology can be realized, at the cochain
level, as a quasi-isomorphism of BDGAs.\end{Th}

\end{document}